\documentclass{amsart}
\usepackage{amssymb}
\usepackage{verbatim}
\usepackage{hyperref}

\newtheorem{theorem}{Theorem}

\newcommand{\Lip}{\mathrm{Lip}}
\newcommand{\R}{\mathbb R}
\newcommand{\E}{\mathbb E}

\begin{document}   
    \title{A Simple Proof of the Johnson--Lindenstrauss Extension Theorem}  
 \author{Manor Mendel}
 \address{Mathematics \& Computer Science Department, The Open University of Israel}
 \email{manorme@openu.ac.il}
\begin{abstract}
    Johnson and Lindenstrauss proved that any Lipschitz mapping from an $n$-point subset of a metric space into Hilbert space can be extended to the whole space, while increasing the Lipschitz constant by  a factor of $O(\sqrt{\log n})$. 
    We present a simplification of their argument that avoids dimension reduction and the Kirszbraun theorem.
\end{abstract} 

\maketitle

%
%
    
The Lipschitz constant of a mapping $f:T\to Y$ between metric spaces is defined as
$\Lip(f)=\sup_{x\ne y \in T} d_Y(f(x),f(y))/d_T(x,y)$.
The following is a Lipschitz extension theorem due to Johnson and Lindenstrauss~\cite{JL84}.

\begin{theorem} \label{thm:JL}
    Let $T$ be an arbitrary $n$-point metric space, and $X\supset T$ an arbitrary superspace. Let $\mathcal H$ be a Hilbert space and $f:T\to \mathcal H$ a   mapping . Then there exists an extension $F:X\to \mathcal H$ of $f$ such that $F|_T=f$, and $\Lip(F) \lesssim \sqrt{\log n} \cdot \Lip(f)$.
\end{theorem}

As was pointed out in~\cite{MN06-ball}, Theorem~\ref{thm:JL} also follows from Ball's 
extension theorem~\cite{Ball} used in conjunction with the Markov type~2 estimates in~\cite{NPSS06}.
The proof of Theorem~\ref{thm:JL} in~\cite{JL84} introduced the celebrated 
Johnson--Lindenstrauss lemma about dimension reduction of finite point sets in Hilbert space. 
In this note we give a simple argument which is a reinterpretation of the original  from~\cite{JL84}
that uses neither dimension reduction nor the Kirszbraun extension theorem.

\begin{proof}[Proof of Theorem~\ref{thm:JL}]
    We may assume without loss of generality that $\mathcal H=\R^n$, i.e.,
    $f:T\to \R^n$.  Furthermore, assume that by rescaling $\Lip(f)=1$. 
    Let $g_1,\ldots, g_n$ be i.i.d. standard Gaussian random variables defined on a probability space $\Omega$, and consider the real-valued Hilbert space 
    $H=L_2(\Omega)$.

    Consider the linear mapping $k:\R^n \to H$ defined by
    \( k(x_1,\ldots,x_n)= \sum_{i=1}^n x_i g_i.
    \)
It is a property of the Gaussian distribution that if $g_1,\ldots,g_n$ are independent standard Gaussian variables, then $\sum_{i=1}^n x_ig_i$ is a centered Gaussian random variable with variance $\sum_{i=1}^nx_i^2$, which means that
$\|\sum_i x_i g_i\|_{L_2(\Omega)}=\|(x_1,\ldots,x_n)\|_2$.
Hence $k$ is an isometric
    isomorphism of $\R^n$ in $k(\R^n) \subset H$.
    Thus, $u:=k\circ f:T \to k(\R^n)$ is also 1-Lipschitz.
    We will use the notation $u=(u_\omega)_{\omega\in \Omega}$, and $k=(k_\omega)_{\omega\in\Omega}$.
    
    Fix $\omega\in \Omega$. 
    Since $u_\omega:T\to \R$,
    by the McShane extension theorem (also known as the nonlinear Hahn--Banach theorem), 
    $u_\omega$ can be extended to  $U_\omega :X\to \mathbb \R$ 
    such that $U_\omega|_T=u_\omega$, and
    $\Lip(U_\omega)=\Lip(u_\omega)$.
    The McShane extension is a simple result. 
    For completeness, we provide a quick proof at the end of this note, see also~\cite[Theorem~1.33]{Weaver-Lip}.
    We thus obtain $U:=(U_\omega)_{\omega\in\Omega}:X\to H$, an extension of $u$.
    Using the orthogonal projection $P:H\to k(\mathbb \R^n)$,    
    we see that $F:=k^{-1}\circ P\circ U:X\to \R^n$ is an extension of $f$. Since 
    $\Lip(k^{-1})=\Lip(P)=1$, we have $\Lip(F) \le \Lip(U)$.
    We next bound $\Lip(U)$. For $x,y\in X$:
    \begin{align}
\nonumber    \| U(x)-U(y)\|_{L_2(\Omega)} &=
    \Bigl (\E_\Omega \bigl [|U_\omega(x)-U_\omega(y)|^2 \bigr ]\Bigr )^{1/2} \\
\nonumber    &\le d(x,y) \Bigl (\E_\Omega [\Lip(U_\omega)^2]\Bigr )^{1/2} \\
\nonumber    & = d(x,y) \Bigl (\E_\Omega [\Lip(u_\omega)^2]\Bigr )^{1/2} \\
\nonumber & =
d(x,y) \left ( 
\E_\Omega \max_{s\ne t\in T} \Bigl(\frac{\sum_i f(s)_i g_i(\omega)- \sum_i f(t)_i g_i(\omega)}{d(s,t)}\Bigr)^2
\right)^{1/2}
\\ 
\nonumber &\le d(x,y)\biggl ( 
\E_\Omega \max_{s\ne t\in T} \Bigl(\sum_{i=1}^n \frac{ f(s)_i-f(t)_i}{\|f(s)-f(t)\|_2} g_i(\omega)\Bigr)^2
\biggr)^{1/2}
\\ \label{eq:max-gauss} &= d(x,y)\Bigl ( 
\E_\Omega \max_{ s\ne t\in T}  G_{f(s)-f(t)}^2 \Bigr)^{1/2},
 \end{align}
where $G_{z}:= \sum_i \frac{z_i}{\|z\|_2} g_i$. $\{G_{f(s)-f(t)}\}_{s,t\in T}$ are standard Gaussians (but not necessarily stochastically independent).
It is a classical and elementary fact that for $m\ge 2$ standard Gaussians $G_1,\ldots, G_m$, with any dependence structure, 
\begin{equation} \label{eq:max-chi-sqaured}
\E [\max\{G_1^2,\ldots, G_m^2\}] \lesssim  \log m.
\end{equation} 
For completeness, we provide a quick proof of~\eqref{eq:max-chi-sqaured} below.
Since $|T|=n$, the maximum  in~\eqref{eq:max-gauss} is over at most $n^2$ squared Gaussians, and therefore~\eqref{eq:max-gauss} is bounded above by a constant times $d(x,y) \cdot \sqrt{\log n}$.
\end{proof}

\begin{proof}[Proof of~\eqref{eq:max-chi-sqaured}]
    By a classical Gaussian tail bound, $\Pr[G_i\ge t] \le e^{-t^2/2}$, and therefore $\Pr[G_i^2\ge t] \le 2e^{-t/2}$.  By the union bound,
$\Pr[ \max\{G_1^2,\ldots,G_m^2\}\ge t]\le 2me^{-t/2}$. 
For any $s\ge 0$, we can estimate $\E[\max\{G_1^2,\ldots,G_m^2\}]$ 
by partitioning the integral up to parameter $s\ge 0$ and from $s$ to $\infty$ as follows:
\begin{multline*}
\E[\max\{G_1^2,\ldots,G_m^2\}] \le s+ \int_s^\infty \Pr \left [\max\{G_1^2,\ldots,G_m^2\} \ge t\right ] dt\\
 \le s+ \int_s^\infty 2m e^{-t/2}dt
 = s+4m e^{-s/2}.
\end{multline*}
Substituting $s=2\log m$ (and noting that $m\ge2$) 
proves~\eqref{eq:max-chi-sqaured}. 
\end{proof}

\begin{proof}[Proof of McShane extension for finite domains]
    Let $T\subset X$,  $T$ finite, and $q:T\to \mathbb R$ $L$-Lipschitz.
    Define $Q:X\to \mathbb R$ by
    \( Q(x) =\min_{t\in T} q(t)+ L d_X(x,t) \).
It is left to the reader to check that $Q|_T=q$  using the Lipschitz condition on $q$.
    Next we prove the Lipshitz condition for $Q$. Fix $x,y\in X$, and let $t\in T$
    satisfy $Q(x)=q(t)+L d_X(x,t)$. Applying the definition of $Q(y)$, and the triangle inequality in $X$, we have
    \[ Q(y)\le q(t)+Ld_X(y,t) \le q(t)+Ld_X(x,t)+Ld_X(x,y)= Q(x)+ Ld_X(x,y).
    \]
    Hence $Q(y)-Q(x)\le Ld_X(x,y)$. By symmetry, we also have
    $Q(x)-Q(y)\le Ld(x,y)$.   
\end{proof}

\medskip 

\paragraph{Acknowledgments.}
    The author thanks  Assaf Naor and the anonymous referees for comments and suggestions that considerably improved the presentation.


\begin{thebibliography}{1}
    
    \bibitem{Ball}
    Ball,~K. (1992).
    \newblock Markov chains, {R}iesz transforms and {L}ipschitz maps.
    \newblock {\em Geom. Funct. Anal.} 2(2): 137--172.
    \newblock  {doi.org/10.1007/BF01896971}
    
    \bibitem{JL84}
    Johnson,~W.~B.,  Lindenstrauss~ J., (1984).
    \newblock Extensions of {L}ipschitz mappings into a {H}ilbert space.
    \newblock In Beals, R., Beck, A., Bellow, A., Hajian, eds., {\em Conference in Modern Analysis and Probability ({N}ew {H}aven,
        {C}onn.,1982)}. Contemporary Mathematics, 26. Providence, RI: American Mathematical Society, pp. 189--206.
    \newblock 
    {doi.org/10.1090/conm/026/737400}
    
    \bibitem{MN06-ball}
    Mendel,~M., Naor, A. (2006).
    \newblock Some applications of {B}all's extension theorem.
    \newblock {\em Proc. Amer. Math. Soc.} 134(9): 2577--2584.
    {doi.org/10.1090/S0002-9939-06-08298-0}
    
    \bibitem{NPSS06}
    Naor,~A., Peres,~Y., Schramm,~O., Sheffield, S. (2006).
    \newblock Markov chains in smooth {B}anach spaces and {G}romov-hyperbolic
    metric spaces.
    \newblock {\em Duke Math. J.}, 134(1): 165--197, 2006.
    \newblock 
    {doi.org/10.1215/S0012-7094-06-13415-4}
    
    \bibitem{Weaver-Lip}
    Weaver,~N. (2018).
    \newblock {\em Lipschitz Algebras}, 2nd ed. Hackensack, NJ: 
    \newblock World Scientific Publishing Co. 
    
\end{thebibliography}

\def\polhk#1{\setbox0=\hbox{#1}{\ooalign{\hidewidth
            \lower1.5ex\hbox{`}\hidewidth\crcr\unhbox0}}}

\end{document}